\documentclass[12pt]{amsart}
\usepackage[all]{xy}
\usepackage[utf8]{inputenc}
\usepackage{tikz-cd}
\usepackage{graphicx}
\usepackage{amsthm}
\usepackage{amstext}
\usepackage{amsmath,amscd,amsfonts}
\usepackage{amssymb}
\usepackage{mathrsfs}
\usepackage{mathtools}
\usepackage[all]{xy}
\usepackage{stmaryrd}
\usepackage{color,comment}
\usepackage{amsmath, amsthm, amsfonts, enumerate}

\DeclareMathOperator{\Ass}{Ass}
\def\N{\mathbb{N}} 
\def\m{\mathfrak{m}}
\def\n{\mathfrak{n}}
\def\gr{\text{gr}} \def\Ie{I_{\epsilon}}
\def\grm{\text{gr}_{\mathfrak m}}
\def\im{{\rm im}}
\def\ker{{\rm ker}}
\def\Me{{M_{\epsilon}}}

\newcommand\sbullet[1][.5]{\mathbin{\vcenter{\hbox{\scalebox{#1}{$\bullet$}}}}}

\newtheorem{theorem}{Theorem}[section]
\newtheorem{lemma}[theorem]{Lemma}
\newtheorem{corollary}[theorem]{Corollary}
\newtheorem{proposition}[theorem]{Proposition}

\newtheorem*{TheoremA}{Theorem A}
\newtheorem*{TheoremB}{Theorem B}
\newtheorem*{TheoremC}{Theorem C}
\theoremstyle{remark}
\newtheorem{remark}[theorem]{\bf Remark}
\theoremstyle{definition}
\newtheorem{example}[theorem]{\bf Example}
\newtheorem{question}[theorem]{\bf Question}

\numberwithin{equation}{section}

\newcommand{\thistheoremname}{}
\newtheorem{genericthm}[theorem]{\thistheoremname}

  \newtheorem*{genericthm*}{\thistheoremname}
\newenvironment{namedthm*}[1]
  {\renewcommand{\thistheoremname}{#1}%
   \begin{genericthm*}}
  {\end{genericthm*}}

\makeatletter
\@namedef{subjclassname@2020}{%
  \textup{2020} Mathematics Subject Classification}
\makeatother
\begin{document}

\title[Local cohomology under small perturbations]{Local cohomology under small perturbations}

\author[L. Duarte]{Lu\'is Duarte}
\address{Dipartimento di Matematica \\
Universit\`a degli Studi di Genova \\
Via Dodecaneso 35\\
I-16146 Genova, Italia}
\email{duarte@dima.unige.it}

\date{\today}

\begin{abstract} 
Let $(R,\m)$ be a Noetherian local ring and $I$ an ideal of $R$. We study how local cohomology modules with support in $\m$ change for small perturbations $J$ of $I$, that is, for ideals $J$ such that $I\equiv J\bmod \m^N$ for large $N$, under the hypothesis that $I$ and $J$ share the same Hilbert function. 
As one of our main results, we show that if $R/I$ is generalized Cohen-Macaulay, then the local cohomology modules of $R/J$ are isomorphic to the corresponding local cohomology modules of $R/I$, except possibly the top one. In particular, this answers a question raised by Quy and V. D. Trung. Our approach also allows us to prove that if $R/I$ is Buchsbaum, then so is $R/J$. Finally, under some additional assumptions, we show that if $R/I$ satisfies Serre's property $(S_n)$, then so does $R/J$.
\end{abstract}


\maketitle

\section{Introduction}

Let $(R,\m)$ be a Noetherian local ring and $I=(f_1,\ldots,f_r)$ be an ideal of $R$. In this paper we study ideals $J$ that satisfy $I\equiv J\bmod \m^N$ for some integer $N>0$, which we call pertubations of $I$. 
A natural way to obtain such ideals is by altering generators of $I$ by elements of $\m^N$, that is, to consider ideals of the form $\Ie=(f_1+\epsilon_1,\ldots,f_r+\epsilon_r)$ with $\epsilon_1,\ldots,\epsilon_r\in\m^N$. This example is particularly relevant when $R$ is a ring of formal power series over a field, and $f_1+\epsilon_1,\ldots,f_r+\epsilon_r$ are polynomial truncations of $f_1,\ldots,f_r$. It then becomes of great importance to know when a given property is preserved under sufficiently small perturbations, as this is related to the problem of trying to approximate analytic singularities by algebraic singularities. 

Many efforts have been made in the literature to determine under which assumptions the two rings $R/J$ and $R/I$ are isomorphic (see \cite{CutkoskySrinivasan}, \cite{cutkosky1997finitedeterminancy}, \cite{greuel2019finite}, \cite{konrad2020criterionequivalence} and \cite{samuel56algebricite}). However, this turns out to be a rather restrictive condition, usually requiring the rings to have very good singularities. For this reason, efforts shifted towards finding conditions that ensure that $R/I$ and $R/J$, even if not isomorphic, still share some relevant attributes.

 In this context, one feature which has been extensively investigated in the literature is the Hilbert function (see \cite{duarte2021betti}, \cite{ma2019filter},  \cite{quy2021when}, \cite{quy2020small}, \cite{srinivas1996invariance} and \cite{srinivas1997finiteness}). The study of its behaviour under small perturbations started with the work of Srinivas and Trivedi in \cite{srinivas1996invariance}. They showed that the Hilbert-Samuel function of a sufficiently small perturbation is at most the original Hilbert-Samuel function (\cite[Lemma 3]{srinivas1996invariance}). Moreover, assuming $R$ is generalized Cohen-Macaulay and $f_1,\ldots,f_r$ is part of a system of parameters of $R$, they showed that for $N\gg 0$ the Hilbert function of $I=(f_1,\ldots,f_r)$ coincides with that of $\Ie=(f_1+\epsilon_1,\ldots,f_r+\epsilon_r)$ for any $\epsilon_1,\ldots,\epsilon_r\in \m^N$ (\cite[Corollary 5]{srinivas1996invariance}). 
The authors asked if this same result holds true whenever $f_1,\ldots,f_r$ is a filter-regular sequence in any Noetherian local ring $R$. This question was later positively answered by Ma, Quy and Smirnov \cite{ma2019filter}. Then, in \cite{quy2021when}, Quy and Ngo Viet Trung were able to find a different approach to the main result of \cite{ma2019filter}, one which allowed them to obtain a specific value for the number $N$ for which the result is valid.
We also point out that in \cite[Theorem 4.5]{duarte2021betti} the author describes another instance in which the Hilbert function is preserved under small perturbations.

In general, the behaviour of a perturbation $J$ of $I$ can be very different from that of $I$. As such, in order to guarantee that $R/J$ shares similar properties to $R/I$, we will impose some assumptions on $R/J$. In this direction, the work of Srinivas and Trivedi leads us to consider perturbations which do not change the Hilbert function. This can be viewed as saying that the singularity of a perturbation is not worse than the original singularity.    
While two ideals with the same Hilbert function do not necessarily share other similar features (such as being Cohen-Macaulay, or generalized Cohen-Macaulay), in the case of perturbations they often do: 
in the main result of \cite{duarte2021betti} the author has shown that, for any $p\in\mathbb N$, any sufficiently small perturbation $J$ of $I$ with the same Hilbert function as $I$ is such that the Betti numbers $\beta_i^R(R/I)$ and $\beta_i^R(R/J)$ coincide for $0\le i\le p$. 

In this paper we continue along this line of investigation by studying how local cohomology modules of $R/I$ with support in the maximal ideal $\m$ are affected by considering sufficiently small perturbations $J$ of $I$, provided $I$ and $J$ share the same Hilbert function. 
Our motivation partially originates from \cite{quy2020small}, where Quy and Van Duc Trung ask the following:
\begin{question}
If $f_1,\ldots,f_r$ is a filter-regular sequence in $R$, does there exist $N>0$ such that, for every $\epsilon_1,\ldots,\epsilon_r\in\m^N$ we have $\ell(H^0_{\m}(R/(f_1,\ldots,f_r)))=\ell(H^0_{\m}(R/(f_1+\epsilon_1,\ldots,f_r+\epsilon_r)))?$
\end{question}
 
Since, by \cite{ma2019filter}, sufficiently small perturbations of filter-regular sequences preserve the Hilbert function, this is just one instance of the problem studied in this paper. In light of this fact, one could even ask the following more general question: given an ideal $I$, does there exist $N>0$ such that $H^0_{\m}(R/I)\cong H^0_{\m}(R/J)$ whenever $J$ is such that $J \equiv I \bmod \m^N$ and $I$ and $J$ have the same Hilbert function?
We provide a positive answer to the question of \cite{quy2020small} in its stronger form. Indeed, we obtain a result which also holds for higher local cohomology modules, provided they are finitely generated. 

\begin{TheoremA}[Theorem \ref{isomorphism}]
Let $(R,\m)$ be a Noetherian local ring and $I$ be an ideal of $R$ such that $H_{\m}^i(R/I)$ are finitely generated for every $i=0,1,\ldots,p$. There exists $N>0$ with the following property: for every ideal $J$ such that $J \equiv I \bmod \m^N$ and such that $I$ and $J$ have the same Hilbert function, we have that
$$H_{\m}^i(R/I)\cong H_{\m}^i(R/J)\quad \forall i=0,1,\ldots,p.$$
\end{TheoremA}

Theorem A significantly extends \cite[Theorem 3.2]{quy2020small} by removing the assumptions that $R$ is generalized Cohen-Macaulay and $I$ is generated by a system of parameters.

The starting point for the proof of Theorem A will be the use of the main result of \cite{duarte2021betti}. When $R/I$ has finite projective dimension, this result allows us to compare the whole free resolution of $R/J$ with that of $R/I$.
We say that a complex $C_\epsilon$ is a perturbation of a complex $C$ if $C_\epsilon$ is obtained by perturbing the maps of $C$ by maps $\epsilon$ with image in a sufficiently large power of $\m$. Perturbations of complexes have been studied for example by Eisenbud in \cite{eisenbud1974adic}, where the homology of $C_\epsilon$ is compared with that of $C$. In fact, throughout this article we will extensively use a slightly improved version of Eisenbud's theorem (see Theorem \ref{T-EisenbudComplexPerturbation}). Other related articles are \cite{aberbach2015uniform} and \cite{eisenbud2005finiteness}, where, assuming $C$ is a free resolution, the authors study the existence of uniform bounds for how small the perturbations of the maps of $C$ need to be so that the complex $C_\epsilon$ is also exact. In these results, however, it is always assumed that the perturbation $C_\epsilon$ of $C$ is already a complex. The main result of \cite{duarte2021betti}, on the other hand, shows that we can in fact build a complex (actually, a free resolution of $R/J$) which is a perturbation of a free resolution of $R/I$.


Under weaker assumptions than those of Theorem A, we show that one can compare the lengths of $H^i_{\m}(R/I)$ and $H^i_{\m}(R/J)$, as well as the annihilators of $H_{\m}^i(R/I)$ and $H^i_\m(R/J)$ (see Theorems \ref{T-LClength} and \ref{PowerM}). 
Theorem A has several important consequences, see Corollary \ref{corollaryl}. For instance, we deduce that properties such as being Cohen-Macaulay or being generalized Cohen-Macaulay are preserved under small enough perturbations, assuming that the ideal and its perturbation share the same Hilbert function.  

Finally, we focus on the property of being Buchsbaum. Thanks to the techniques developed in this paper we are able to compare the truncated normalized dualizing complexes of $R/I$ and $R/J$ provided $R/I$ is generalized Cohen-Macaulay, showing that they are isomorphic in the derived category. As a consequence, we deduce that the Buchsbaum property is also preserved under sufficiently small pertubations which preserve the Hilbert function.

\begin{TheoremB}[Theorem \ref{Buchsbaum}]
Let $(R,\m)$ be a Noetherian local ring and $I$ be an ideal of $R$ such that $R/I$ is generalized Cohen-Macaulay. There exists $N>0$ with the following property: for every ideal $J$ such that $J \equiv I \bmod \m^N$ and such that $I$ and $J$ have the same Hilbert function, we have that $R/I$ is Buchsbaum if and only if $R/J$ is Buchsbaum. 
\end{TheoremB}

We also show that, under some additional assumptions, Serre's properties $(S_n)$ are also preserved under small perturbations which preserve the Hilbert function. 

\begin{TheoremC}[Theorem \ref{serreperturbation}]
Suppose $(R,\m)$ is excellent. Let $I$ be an ideal of $R$ for which $R/I$ is formally equidimensional. There exists $N>0$ with the following property: if $J$ is an ideal of $R$ for which $I\equiv J\bmod \m^N$, $R/J$ is formally equidimensional and $R/J$ has the same Hilbert function as $R/I$, then for every $n\ge 0$ if $R/I$ satisfies Serre's property $(S_n)$ then so does $R/J$.
\end{TheoremC}




\section{Preliminaries}\label{preliminaries}

In this section we will introduce some preliminary results and notation which will be used throughout this article.

$(R,\m)$ will denote a Noetherian local ring and $k=R/\m$ will denote its residue field. An element $f \in R$ is called \emph{filter-regular} if $\Ass_R((0\colon (f))) \subseteq \{\m\}$. A sequence of elements $f_1,\ldots,f_r$ of $R$ is called a \emph{filter-regular sequence} if, for every $1 \leq i \leq r$, the image of $f_i$ in $R/(f_1,\ldots,f_{i-1})$ is a filter-regular element. 

Given an $R$-module $M$, we will denote its length by $\ell(M)$. We will use $H_{\m}^i(M)$ to denote the $i$-th local cohomology module of $M$ with respect to the maximal ideal $\m$. 
The ring $R$ is said to be \emph{generalized Cohen-Macaulay} if the local cohomology modules $H_{\m}^i(R)$ have finite length for every $i<\dim(R)$.

The \emph{Hilbert function} $\text{HF}_R\colon \mathbb Z_{\ge 0} \to \mathbb Z_{\ge 0}$ of $R$ is defined to be the Hilbert function of its associated graded ring $\gr_\m(R) = \bigoplus_{i=0}^\infty \m^i/\m^{i+1}$, which is a standard graded $k$-algebra. In other words, $\text{HF}_R$ is given by  $\text{HF}_R(i)=\dim_k(\m^i/\m^{i+1})$ for every $i\ge 0$.

\subsection{Initial modules and ideals: } An \emph{$\m$-filtration} $\mathbb M$ of $M$ is a collection $\{M_i\}_{i\ge 0}$ of submodules of $M$ such that $\m M_i\subseteq M_{i+1}\subseteq M_{i}$ for every $i\ge 0$. 
The filtration $\mathbb M$ is called \emph{stable} (or \emph{good}) if $\m M_n=M_{n+1}$ for all sufficiently large $n$. Given an $\m$-filtration $\mathbb M$, we define the 
\emph{associated graded module} of $M$ with respect to $\mathbb M$ as 
$$\gr_{\mathbb M}(M)=\bigoplus_{i=0}^{\infty}{\frac{M_i}{M_{i+1}}}.$$

Since $\mathbb M$ is an $\m$-filtration, $\gr_{\mathbb M}(M)$ has a natural structure of a graded module over $\grm(R)$.
In case $\mathbb M=\{\m^iM\}_{i\ge 0}$ is the $\m$-adic filtration of $M$, we will denote $\gr_{\mathbb M}(M)$ simply by $\grm(M)$. 


If $N$ is a submodule of 
$M$, we define the \emph{initial module} of $N$, denoted by $N^*$, as the kernel of the graded map of $\grm(R)$-modules 
$$\grm(M)\to \grm(M/N)$$ 
induced by the projection $M\to M/N$, so that $\grm(M/N)\cong \grm(M)/N^*$. We thus have that

$$N^*=\bigoplus_{i=0}^{\infty}{\frac{N\cap \m^iM+\m^{i+1}M}{\m^{i+1}M}} \cong \bigoplus_{i=0}^{\infty}{\frac{N\cap \m^iM}{N\cap \m^{i+1}M}}.$$
Therefore $N^*$ can also be viewed as the associated graded module of $N$ with respect to the filtration $\{N\cap \m^iM\}_{i\ge 0}$. For another description of $N^*$, in terms of initial forms, we refer the reader to \cite{duarte2021betti}.
Whenever we write $N^*$, the inclusion $N\subseteq M$ will be implicit from the context. 

\begin{remark}
In what follows, we will refer to the Hilbert function of an ideal $I \subseteq R$ to mean the Hilbert function of the local ring $R/I$, that is, the Hilbert function of $\gr_{\m/I}(R/I)$. Thanks to the isomorphism described above, it coincides with the Hilbert function of $\gr_\m(R)/I^*$. 
\end{remark}


If $I$ is an ideal of $R$, and $N\subseteq M$ are finitely generated $R$-modules, then by the Artin-Rees lemma there exists a positive number $s$ such that
$$I^nM\cap N=I^{n-s}(I^sM\cap N)\quad \text{for every } n\ge s.$$
We will use $\text{AR}(I,N\subseteq M)$ to denote this Artin-Rees number, that is, the smallest such $s$. 

We thus have that the filtration $\{\m^iM\cap N\}_{i\ge 0}$ of $N$ is stable. Moreover,
according to \cite[Proposition 2.1 (a)]{herzog2016homology}, the number $\text{AR}(\m,N\subseteq M)$ is related to the initial module $N^*$ of $N$. More precisely, $\text{AR}(\m,N\subseteq M)$ coincides with the largest degree of an element in a minimal set of homogeneous generators of $N^*$ as a $\gr_\m(R)$-module. In particular, if we have finitely generated modules $N_1\subseteq M_1$ and $N_2\subseteq M_2$ such that $N_1^*$ and $N_2^*$ are isomorphic as $\gr_\m(R)$-modules, then $\text{AR}(\m,N_1\subseteq M_1)=\text{AR}(\m,N_2\subseteq M_2)$.

The following result states that taking initial modules behaves well with respect to taking quotients. It can be shown as in \cite[Lemma 2.1]{quy2021when}.

\begin{proposition}\label{quotientofstars}
Let $L\subseteq N\subseteq M$ be finitely generated $R$-modules. Then $$\left(\frac{N}{L}\right)^*\cong\frac{N^*}{L^*}.$$
Here $\left(\frac{N}{L}\right)^*$ is computed inside $\emph{\text{gr}}_{\m}(\frac ML)\cong \frac{\emph{\text{gr}}_{\m}(M)}{L^*}$, while $N^*$ and $L^*$ are computed inside $\emph{\text{gr}}_{\m}(M)$.
\end{proposition}

\subsection{Approximations of complexes:} We now discuss a result which will be crucial in the proof of our main results ahead. 
Let
$$\begin{tikzcd}
C\colon  \cdots \arrow[r] & C_{n+1} \arrow[r, "f_{n+1}"] & C_{n} \arrow[r, "f_n"] & C_{n-1} \arrow[r] & \cdots \end{tikzcd}$$ 
be a complex of finitely generated $R$-modules.
An $\m$\emph{-adic approximation of} $C $ of order $d=(\ldots,d_{n+1},d_n,d_{n-1},\ldots)\in\mathbb N^{\mathbb Z}$ is a complex $C _{\epsilon}$ of the form $$\begin{tikzcd}[row sep=large,column sep = large]
C _{\epsilon}\colon  \cdots \arrow[r] & C_{n+1} \arrow[r, "f_{n+1}+\epsilon_{n+1}"] & C_{n} \arrow[r, "f_n+\epsilon_n"] & C_{n-1} \arrow[r] & \cdots\end{tikzcd}$$
where $\epsilon_n$ is a map from $C_n$ to $\m^{d_n}C_{n-1}$ for all $n$.

For any integer $n$, let $H_n(C)$ denote the $n$-th homology module of $C$. We will consider the initial module $H_n(C )^* = \bigoplus_p H_n(C)^*_p$ of $H_n(C)$ with respect to the inclusion $H_n(C )\subseteq C_n/\text{im }(f_{n+1})$, where the latter has a natural $\m$-adic filtration induced by $\{\m^i C_n\}_{i \ge 0}$, and we want to compare such an initial module with the one obtained from a perturbed complex $C_\epsilon$. We warn the reader that the initial module of $H_n(C)$ is computed inside $\gr_\m(C_n/{\rm im}(f_{n+1}))$, while the initial module of $H_n(C_\epsilon)$ is computed inside $\gr_\m(C_n/{\rm im}(f_{n+1}+\epsilon_{n+1}))$. 
Despite this difference, the following theorem, which consists of a slight improvement of the main result of \cite{eisenbud1974adic}, shows a close relation between them.

\begin{theorem}\label{T-EisenbudComplexPerturbation}
Let $$\begin{tikzcd}
C\colon  \cdots \arrow[r] & C_{n+1} \arrow[r, "f_{n+1}"] & C_{n} \arrow[r, "f_n"] & C_{n-1} \arrow[r] & \cdots \end{tikzcd}$$ be a complex of finitely generated $R$-modules. There exists a sequence of integers $d=(\ldots,d_{n+1},d_n,d_{n-1},\ldots)$ such that, whenever
$$\begin{tikzcd}[row sep=large,column sep = large]
C _{\epsilon}\colon  \cdots \arrow[r] & C_{n+1} \arrow[r, "f_{n+1}+\epsilon_{n+1}"] & C_{n} \arrow[r, "f_n+\epsilon_n"] & C_{n-1} \arrow[r] & \cdots\end{tikzcd}$$ is a complex which is an $\m$-adic approximation of $C$ of order $d=(\ldots,d_{n+1},d_n,d_{n-1},\ldots)$, we have that
\begin{itemize}
\item[(i)] $H_n(C_\epsilon )^*_p$ is a subquotient of $H_n(C )^*_p$ for all $n$ and all $p$.
\item[(ii)] If $H_n(C )$ and $H_{n-1}(C )$ are both annihilated by some power of $\m$, then $H_n(C )^*\cong H_n(C_\epsilon )^*$. Moreover, in this case we have $\emph{\text{im}}(f_{n+1})^*= \emph{\text{im}}(f_{n+1}+\epsilon_{n+1})^*$ and $\emph{\text{ker}}(f_{n})^*\cong \emph{\text{ker}}(f_{n}+\epsilon_{n})^*$, where all these modules are computed inside $\emph{\text{gr}}_{\m}(C_n)$.
\end{itemize}

\end{theorem}


The statements in Theorem \ref{T-EisenbudComplexPerturbation} are shown in \cite{eisenbud1974adic} and \cite[Theorem 2.4]{duarte2021betti}, except for the second claim made in point (ii). We now prove this claim.
\begin{proof}
Applying point (i) to the complex $\begin{tikzcd}
0 \arrow[r] & C_n \arrow[r, "f_n"] & C_{n-1} \arrow[r] & 0
\end{tikzcd}$ we obtain that, for $d_n$ large enough, $\ker(f_n+\epsilon_n)^*$ is isomorphic to a graded subquotient of $\ker(f_n)^*$. Moreover, since $\im(f_{n+1}+\epsilon_{n+1})\equiv \im(f_{n+1})\bmod \m^{d_{n+1}}C_n$, by \cite[Proposition 3.1]{duarte2021betti}, for $d_{n+1}$ large enough we have that $\im(f_{n+1})^*\subseteq \im(f_{n+1}+\epsilon_{n+1})^*$. In particular, $\dim_k(\ker(f_n)^*_p)\ge\dim_k(\ker(f_n+\epsilon_n)^*_p)$ and  $\dim_k(\im(f_{n+1})^*_p)\le\dim_k(\im(f_{n+1}+\epsilon_{n+1})^*_p)$ for every $p\ge 0$. Also, since for $d_n$ and $d_{n+1}$ large enough we know that $H_n(C)^*\cong H_n(C_{\epsilon})^*$, by Proposition \ref{quotientofstars} we have that 
$$\frac{\ker(f_n)^*}{\im(f_{n+1})^*}\cong \left(\frac{\ker(f_n)}{\im(f_{n+1})}\right)^*\cong \left(\frac{\ker(f_n+\epsilon_n)}{\im(f_{n+1}+\epsilon_{n+1})}\right)^*\cong \frac{\ker(f_n+\epsilon_n)^*}{\im(f_{n+1}+\epsilon_{n+1})^*}. $$
Hence
$$\dim_k(\ker(f_n)^*_p)-\dim_k(\im(f_{n+1})^*_p)=\dim_k(\ker(f_n+\epsilon_n)^*_p)-\dim_k(\im(f_{n+1}+\epsilon_{n+1})^*_p)$$
for every $p\ge 0$. This holds if and only if $\dim_k(\ker(f_n)^*_p)=\dim_k(\ker(f_n+\epsilon_n)^*_p)$ and also $\dim_k(\im(f_{n+1})^*_p)=\dim_k(\im(f_{n+1}+\epsilon_{n+1})^*_p)$ for every $p\ge 0$, that is, if and only if $\ker(f_{n})^*\cong \ker(f_{n}+\epsilon_{n})^*$ and $\im(f_{n+1})^*= \im(f_{n+1}+\epsilon_{n+1})^*$.
\end{proof}

We will also need the following:

\begin{proposition}\label{L-imkerstar}
Let $f\colon C\to D$ be a map of finitely generated $R$-modules. 
Set $s=\emph{\text{AR}}(\m,\im(f)\subseteq D)$ and $r=\emph{\text{AR}}(\m,\ker(f)\subseteq C)$. Then for every $N>s+r$ and every map $\epsilon\colon C\to \m^ND$ the following are equivalent:
\begin{itemize}
\item[(i)] $\im(f)^*=\im(f+\epsilon)^*$.
\item[(ii)] $\ker(f)^* \cong \ker(f+\epsilon)^*$.
\item[(iii)] $\ker(f)^*=\ker(f+\epsilon)^*$.
\item[(iv)] $\ker(f)\equiv \ker(f+\epsilon) \bmod \m^{N-s}C$.
\end{itemize}
\end{proposition}

To prove it, we will require a lemma. We will denote by $\mu(M)$ the minimal number of generators of an $R$-module $M$. 

\begin{lemma}\label{L-NumberGeneratorsInequality}
Let $(R,\m)$ be a Noetherian local ring and let $M\subseteq F$ be finitely generated $R$-modules. For $N>\emph{\text{AR}}(\m,M\subseteq F)$ the following holds: if $\Me$ is a submodule of $F$ such that $M\equiv \Me \bmod \m^NF$, then 
\begin{itemize}
\item[(i)]$\mu(\m^nM)\le\mu(\m^n\Me)$ for every $n\ge 0$.
\item[(ii)] If $\mu(\m^nM)=\mu(\m^n\Me)$ for some $n\ge 0$ and $\{f_1,\ldots,f_r\}$ is a minimal generating set of $\m^nM$, then there exist $\epsilon_1,\ldots,\epsilon_r\in\m^{N+n}F$ such that $\{f_1+\epsilon_1,\ldots,f_r+\epsilon_r\}$ is a minimal generating set of $\m^n\Me$.
\item[(iii)] If $\mu(\m^nM)=\mu(\m^n\Me)$ for all $n\ge 0$, then $M^*=\Me^*$.
\end{itemize}
\end{lemma}
\begin{proof}
Let $s=\text{AR}(\m,M\subseteq F)$ and let $\{f_1,\ldots,f_r\}$ be a minimal generating set of $\m^nM$. As $\m^nM\equiv \m^n\Me \bmod \m^{N+n}F$, there are $\epsilon_1,\ldots,\epsilon_r\in\m^{N+n}F$ such that $f_i+\epsilon_i\in \m^n\Me$ for $i=1,\ldots,r$. To prove (i), as well as (ii), it is enough to show that $\{f_1+\epsilon_1,\ldots,f_r+\epsilon_r\}$ is part of a minimal generating set of $\m^n\Me$. By Nakayama, this is equivalent to proving that the image of $\{f_1+\epsilon_1,\ldots,f_r+\epsilon_r\}$ in $\m^n\Me/\m^{n+1} \Me$ is an $R/\m$-linearly independent set; so supposing $a_1,\ldots,a_r\in R$ are such that $\sum_{i=1}^{r}{a_i(f_i+\epsilon_i)}\in\m^{n+1} \Me$, we need to show that $a_1,\ldots,a_r\in\m$. Indeed, this implies that
\begin{equation*}
\begin{split}
\sum_{i=1}^{r}{a_if_i}\in \m^nM\cap\left(\m^{n+1} \Me+ \left(\sum_{i=1}^{r}{a_i\epsilon_i}\right)\right)& \subseteq \m^nM\cap(\m^{n+1}\Me+\m^{N+n}F)
\\ & = \m^nM\cap(\m(\m^n\Me+\m^{N+n-1}F))
\\ & = \m^nM\cap(\m(\m^nM+\m^{N+n-1}F))
\\ & = \m^{n+1}M+\m^{N+n}F\cap \m^nM
\\ & \subseteq \m^{n+1}M+ \m^{N+n}\cap M
\\ & \subseteq \m^{n+1}M+ \m^{N+n-s} M
\\ & = \m^{n+1}M
\end{split}
\end{equation*}
Because $\{f_1,\ldots,f_r\}$ is a minimal system of generators of $M$, we conclude that $a_1,\ldots,a_r\in\m$.

We now prove (iii). 
First of all we will show that 
\begin{equation}\label{weakAR}\m^nF\cap \Me\subseteq \m^{n-s}\Me\quad\text{for all }n\ge s.
\end{equation}
Let $n\ge s$ and suppose $x$ is a nonzero element in $\m^nF\cap \Me$. We will show that $x\in\m^{n-s}\Me$. Let $l$ be the largest natural number for which $a\in\m^l\Me$. By means of contradiction, suppose $l<n-s$. Let $\{f_1,\ldots,f_r\}$ be a minimal generating set of $\m^lM$. Then by (ii) there are $\epsilon_1,\ldots,\epsilon_r\in\m^{N+l}F$ such that $\{f_1+\epsilon_1,\ldots,f_r+\epsilon_r\}$ is a minimal generating set of $\m^l\Me$. Writing $x=a_1(f_1+\epsilon_1)+\cdots + a_r(f_r+\epsilon_r)$ we observe that
\begin{equation*}
\begin{split}
\sum_{i=1}^{r}{a_if_i}  = x-\sum_{i=1}^{r}{a_i\epsilon_i}  & \in (\m^nF+\m^{N+l}F)\cap M
\\ & \subseteq \m^{l+s+1}F\cap M
\\ & \subseteq \m^{l+1}M.
\end{split} 
\end{equation*}
Hence all $a_1,\ldots,a_r$ must belong to $\m$. But this implies that $a\in\m^{l+1}\Me$, which contradicts our choice of $l$. We have thus shown (\ref{weakAR}). We now prove that $M^*=\Me^*$. By \cite[Proposition 3.1]{duarte2021betti} we have $M^*\subseteq \Me^*$. It remains to prove the reverse inclusion. Since 
$$M^*=\bigoplus_{i=0}^{\infty}{\frac{M\cap \m^iF+\m^{i+1}F}{\m^{i+1}F}}\quad\text{and}\quad \Me^*=\bigoplus_{i=0}^{\infty}{\frac{\Me\cap \m^iF+\m^{i+1}F}{\m^{i+1}F}},$$
it is enough to show that $$\Me\cap \m^iF\subseteq M\cap \m^iF+\m^{i+1}F=(M+\m^{i+1}F)\cap \m^iF \quad\text{for all } i\ge 0.$$ Indeed, this clear if $i<N$, while if $i\ge N$ we have by (\ref{weakAR}) that $\Me\cap \m^iF\subseteq \m^{i-s}\Me\subseteq \m^{i-s}(M+\m^NF)\subseteq M+\m^{N+i-s}F \subseteq M+\m^{i+1}F$.
\end{proof}

\begin{proof}[Proof of Proposition \ref{L-imkerstar}]
(i) $\Rightarrow$ (iv): For every $x\in \ker(f+\epsilon)$ we have that  
$f(x)=-\epsilon(x) \in $ $ \m^{N}D\cap\im(f) 
\subseteq \m^{N-s}\im(f),$
which implies that there exists $x'\in \m^{N-s}C$ such that $f(x)=f(x')$, that is, $x-x'\in \ker(f)$. This shows that $\ker(f+\epsilon)\subseteq \ker(f)+\m^{N-s}C$.        
Since the hypothesis $\im(f)^*=\im(f+\epsilon)^*$ implies that $\text{AR}(\m,\im(f+\epsilon)\subseteq D)=s$, the proof that that $\ker(f)\subseteq \ker(f+\epsilon)+\m^{N-s}C$ follows by the same argument.

(iv) $\Rightarrow$ (iii): Since $N-s>r$, by \cite[Proposition 3.1]{duarte2021betti} we have that $\ker(f)^*\subseteq \ker(f+\epsilon)^*$. On the other hand, by Theorem \ref{T-EisenbudComplexPerturbation} applied to the complex $\begin{tikzcd}
0 \arrow[r] & C \arrow[r, "f"] & D \arrow[r] & 0
\end{tikzcd}$, $\ker(f+\epsilon)^*$ is isomorphic to a graded subquotient of $\ker(f)^*$. Hence we must have the equality $\ker(f)^*=\ker(f+\epsilon)^*$.

(iii) clearly implies (ii). We now show (ii) $\Rightarrow$ (i): Since
$$\grm(\im(f))\cong\grm\left(\frac{C}{\ker(f)}\right)\cong \frac{\grm(C)}{\ker(f)^*},$$
for every $n\ge 0$ we have that
\begin{align*}
\mu(\m^n\im(f))&=\dim_k(\grm(\im(f))_n)=\dim_k\left(\frac{\grm(C)}{\ker(f)^*}\right)_n =\dim_k\left(\frac{\grm(C)_n}{\ker(f)^*_n}\right) \\
&=\dim_k(\grm(C)_n)-\dim_k(\ker(f)^*_n)
\end{align*}
and similarly with $f$ replaced by $f+\epsilon$. Since $\ker(f)^*\cong \ker(f+\epsilon)^*$ by hypothesis, we have $\dim_k(\ker(f)^*_n)=\dim_k(\ker(f+\epsilon)^*_n)$ for all $n\ge 0$. We thus conclude that $\mu(\m^n\im(f))=\mu(\m^n\im(f+\epsilon))$ for all $n\ge 0$. As $\im(f)\equiv \im(f+\epsilon)\bmod\m^ND$, it now follows by Lemma \ref{L-NumberGeneratorsInequality} that $\im(f)^*=\im(f+\epsilon)^*$.
\end{proof}

\section{Local cohomology under small perturbations}\label{SectionBetti}




We start by setting up the notation that we will be using for the rest of the article. 
Let $I$ be a fixed ideal of $R$ and let $J$ be an ideal such that $J \equiv I \bmod \m^N$ and such that $I$ and $J$ have the same Hilbert function. 

\begin{remark}\label{Remark-filterregular}
We recall that, by \cite{ma2019filter}, if $I=(f_1,\ldots,f_r)$, where $f_1,\ldots,f_r\in R$ is a filter-regular sequence in $R$, then, for $N\gg 0$, $I$ and $J$ have the same Hilbert function for every ideal $J$ of the form $(f_1+\epsilon_1,\ldots,f_r+\epsilon_r)$, where $\epsilon_1,\ldots,\epsilon_r$ are any elements in $\m^N$.
\end{remark}

Let $\widehat {(-)}$ denote $\m$-adic completion. From $I+\m^N=J+\m^N$ we get $(I+\m^N)\widehat R=(J+\m^N)\widehat R$, that is, $\widehat {I}\equiv \widehat {J} \bmod \widehat {\m}^N$. 
Also, by \cite[Proposition 3.2]{duarte2021betti} and \cite[Theorem 7.1]{eisenbud2013commutative} we have that $\gr_{\widehat{\m}}(\widehat{R}/\widehat{I})=\gr_{\m}({R/I})\cong \gr_{\m}({R/J})=\gr_{\widehat{\m}}(\widehat{R}/\widehat{J})$, hence $\widehat{R}/\widehat{I}$ and $\widehat{R}/\widehat{J}$ have the same Hilbert function.
Since $H^i_{\m}(R/I)\cong H^i_{\widehat{\m}}(\widehat{R}/\widehat{I})$ and $H^i_{\m}(R/J)\cong H^i_{\widehat{\m}}(\widehat{R}/\widehat{J})$ for all $i$, in the proofs of all the results in this section we will always be able to assume that $R$ is complete. 

With this new assumption, by the Cohen structure theorem, $R$ is a quotient ring of a complete regular local ring $S$. Let $\n$ be the maximal ideal of $S$ and $L$ be the ideal of $S$ such that $R=S/L$. Let also $I_0,J_0\supseteq L$ be the ideals of $S$ such that $I=I_0/L$ and $J=J_0/L$.
Since $I+\m^N=J+\m^N$, that is, $\frac{I_0}{L}+\frac{\n^N+L}{L}=\frac{J_0}{L}+\frac{\n^N+L}{L}$, it follows that $I_0\equiv J_0\bmod \n^N$. Also, by hypothesis $S/I_0=R/I$ and $S/J_0=R/J$ have the same Hilbert function.

Let $d=\dim(S)$. 
As $S$ is regular, there is a finite free resolution
 $$\begin{tikzcd}
 F_{\sbullet}\colon  0 \arrow[r]  & F_d \arrow[r, "f_d"] & F_{d-1} \arrow[r, "f_{d-1}"] & \cdots \arrow[r, "f_2"] & F_1 \arrow[r, "f_1"] & F_0 
\end{tikzcd}$$
 of $R/I$ as an $S$-module. Here, $F_i=0$ and $f_i=0$ if $i>\text{pd}_S(R/I)$. By \cite[Theorem 3.6]{duarte2021betti}, given any $N_0\in\N$, for $N\gg 0$ there exists a minimal free resolution of $R/J$ of the form
$$\begin{tikzcd}[row sep=large,column sep = large]
 F^{\epsilon}_{\sbullet}\colon  0 \arrow[r]  & F_d \arrow[r, "f_d+\epsilon_{d}"] & F_{d-1} \arrow[r, "f_{d-1}+\epsilon_{d-1}"] & \cdots \arrow[r, "f_2+\epsilon_{2}"] & F_1 \arrow[r, "f_1+\epsilon_{1}"] & F_0 
\end{tikzcd}$$
where $\epsilon_i(F_i)\subseteq\n^{N_0}F_{i-1}$ for every $i=1,\ldots,d$.

We will use 
$(-)^{\vee}$ to denote the functor $\text{Hom}_S(-,E_S(k))$, where $E_S(k)$ is the injective hull of the residue field $k=S/\n$ as an $S$-module.
By local duality (\cite[Appendix A4.2]{eisenbud2013commutative}) we have that 
\begin{equation}\label{dualityEq}
H^i_{\m}(R/I)^{\vee}\cong \text{Ext}_S^{d-i}(R/I,S)\cong H^{d-i}(\text{Hom}_S(F_{\sbullet},S))\quad\text{for all } i
\end{equation} 
and
\begin{equation}\label{dualityEq}
H^i_{\m}(R/J)^{\vee}\cong \text{Ext}_S^{d-i}(R/J,S)\cong H^{d-i}(\text{Hom}_S(F^{\epsilon}_{\sbullet},S))\quad\text{for all } i.
\end{equation} 

Notice that $\text{Hom}_S(F^{\epsilon}_{\sbullet},S)$ is an $\n$-adic approximation of the cochain complex $\text{Hom}_S(F_{\sbullet},S)$ of order $(\ldots,N_0,N_0, N_0,\ldots)$.


\begin{theorem}\label{T-LClength}
Let $(R,\m)$ be a Noetherian local ring and let $I$ be an ideal of $R$. There exists $N>0$ with the following property: for every ideal $J$ such that $J \equiv I \bmod \m^N$ and such that $I$ and $J$ have the same Hilbert function, one has that 
\begin{itemize}
\item[(i)] if $H^i_{\m}(R/I)$ is finitely generated, then $H^i_{\m}(R/J)$ is also finitely generated and $$\ell(H^i_{\m}(R/I))\ge \ell(H^i_{\m}(R/J)).$$
\item[(ii)] if $H^i_{\m}(R/I)$ and $H^{i-1}_{\m}(R/I)$ are both finitely generated, then $$\ell(H^i_{\m}(R/I))=\ell(H^i_{\m}(R/J)).$$
\end{itemize} 
\end{theorem}

\begin{proof}
Let $M= H^{d-i}(\text{Hom}_S(F_{\sbullet},S))$ and $K= H^{d-i}(\text{Hom}_S(F^{\epsilon}_{\sbullet},S))$. 

By point (i) of Theorem \ref{T-EisenbudComplexPerturbation} we have that, given that $N_0\gg 0$, the initial module $K^*$ of $K$ is a graded subquotient of the initial module $M^*$ of $M$. It follows that 
$\ell(M)=\sum_{p\ge0}{\dim_k(M^*_p)} \ge \sum_{p\ge0}{\dim_k(K^*_p)}=\ell(K).$
Assuming $H_{\m}^i(R/I)$ is finitely generated, we then have that $H_{\m}^i(R/I)^{\vee}\cong M$ has finite length and $\ell(H_{\m}^i(R/I))=\ell(H_{\m}^i(R/I)^{\vee})=\ell(M)$. Since $\ell(M)\ge \ell(K)$, it follows that $K$ has finite length and $\ell(K)=\ell(K^{\vee})=\ell(H_{\m}^i(R/J))$. This shows (i).

If in adition we assume that $H^{i-1}_{\m}(R/I)$ is finitely generated, we then have that $H^{i-1}_{\m}(R/I)^{\vee}\\ \cong H^{d-i+1}(\text{Hom}_S(F_{\sbullet},S))$ has finite length.  This means that there is a power of $\n$ that annihilates both $H^{d-i}(\text{Hom}_S(F_{\sbullet},S))$ and $H^{d-i+1}(\text{Hom}_S(F_{\sbullet},S))$. By point (ii) of Theorem \ref{T-EisenbudComplexPerturbation} we obtain that, given $N_0\gg 0$, $K^*\cong M^*$. We conclude that 
$\ell(H_{\m}^i(R/I))=\ell(M)=\sum_{p\ge0}{\dim_k\left(M^*_p\right)} =\sum_{p\ge0}{\dim_k\left(K^*_p\right)}=\ell(K)=\ell(H_{\m}^i(R/J)).$
\end{proof}

Given an $R$-module $M$, for each $i$ we  define $$\ell_{\m}^i(M)=\inf\{n\in\mathbb N \colon \m^nH_{\m}^i(M)=0\}.$$

Next we show that the assumptions of Theorem \ref{T-LClength} (ii) are also enough for us to able to compare $\text{Ann}(H_{\m}^i(R/J))$ with $\text{Ann}(H_{\m}^i(R/I))$. 

\begin{theorem}\label{PowerM}
Let $(R,\m)$ be a Noetherian local ring and $I$ be an ideal of $R$ for which $H_{\m}^i(R/I)$ and $H_{\m}^{i-1}(R/I)$ are finitely generated. There exists $N>0$ with the following property: for every ideal $J$ such that $J \equiv I \bmod \m^N$ and such that $I$ and $J$ have the same Hilbert function, we have 
$$\emph{\text{Ann}}(H_{\m}^i(R/J)) = \emph{\text{Ann}}(H_{\m}^i(R/I)).$$ 
\end{theorem}
\begin{proof}
Let $(-)'$ denote the functor $\text{Hom}_S(-,S). $ For simplicity of notation, rewrite the piece 
$$\begin{tikzcd}
 F_{d-i-1}' \arrow[r, "f_{d-i}'"] & F_{d-i}' \arrow[r, "f_{d-i+1}'"] & F_{d-i+1}' \arrow[r, "f_{d-i+2}'"] & F_{d-i+2}' \arrow[r, "f_{d-i+3}'"] & F_{d-i+3}' 
\end{tikzcd} $$
of $\text{Hom}_S(F_{\sbullet},S)$ and the piece
$$ \begin{tikzcd}[row sep=large,column sep = large]
  F_{d-i-1}' \arrow[r, "f_{d-i}'+\epsilon_{d-i}'"] & F_{d-i}' \arrow[r, "f_{d-i+1}'+\epsilon_{d-i+1}'"] & F_{d-i+1}' \arrow[r, "f_{d-i+2}'+\epsilon_{d-i+2}'"] & F_{d-i+2}' \arrow[r, "f_{d-i+3}'+\epsilon_{d-i+3}'"] & F_{d-i+3}'
\end{tikzcd} $$
of $\text{Hom}_S(F^{\epsilon}_{\sbullet},S)$ respectively as

\begin{equation}\label{complex1}
 \begin{tikzcd}[row sep=large,column sep = large]
  G \arrow[r, "d"] & F \arrow[r, "h"] & H \arrow[r, "u"] & U \arrow[r, "v"] & V
\end{tikzcd}
\end{equation}
and
\begin{equation}\label{complex2}
 \begin{tikzcd}[row sep=large,column sep = large]
  G \arrow[r, "d+\delta"] & F \arrow[r, "h+\epsilon"] & H \arrow[r, "u+\iota"] & U \arrow[r, "v+\gamma"] & V.
\end{tikzcd}
\end{equation}
We recall that $\delta$, $\epsilon$, $\iota$ and $\gamma$ are such that $\im(\delta)\subseteq \m^{N_0}F$, $\im(\epsilon)\subseteq \m^{N_0}H$, $\im(\iota)\subseteq \m^{N_0}U$ and $\im(\gamma)\subseteq \m^{N_0}V$, where $N_0$ is a natural number which can be sufficiently large, given that $N$ is made large enough.

Assuming $H_{\m}^i(R/I)$ is finitely generated, by Theorem \ref{T-LClength} we can assume $N$ is large enough so that $H_{\m}^i(R/J)$ is also finitely generated. Let $l=\ell_{\m}^i(R/I)$ and $l'=\ell_{\m}^i(R/J)$. $l$ coincides with the least integer $p$ for which $\n^p$ annihilates $H_{\m}^i(R/I)^{\vee} \cong\frac{\ker(h)}{\im(d)}$, while $l'$ coincides with the least integer $p$ for which $\n^p$ annihilates $H_{\m}^i(R/J)^{\vee}\cong\frac{\ker(h+\epsilon)}{\im(d+\delta)}$.

Moreover, since we are assuming $H_{\m}^i(R/I)$ and $H_{\m}^{i-1}(R/I)$ are finitely generated, the cohomologies of (\ref{complex1}) at $F$ and $H$ are annihilated by some power of $\n$. By Theorem \ref{T-EisenbudComplexPerturbation} (ii) applied to (\ref{complex1}) and (\ref{complex2}) it follows that $\ker(h)^*\cong \ker(h+\epsilon)^*$. As such, let $r=\text{AR}(\n,\ker(h)\subseteq F)=\text{AR}(\n,\ker(h+\epsilon)\subseteq F)$ and let $s=\text{AR}(\n,\im(h)\subseteq H)$. By Proposition \ref{L-imkerstar}, if $N_0$ is large enough we also have 
$\ker(h)\equiv \ker(h+\epsilon) \bmod \n^{N_0-s}F$. 

Next we show that $l'\le l$. Since $\n^l\ker(h)\subseteq \im(d)$ and using the fact that $\im(d)\equiv\im(d+\delta)\bmod \n^{N_0}F$, we observe that 
\begin{align*}
\n^{l}\ker(h+\epsilon) & \subseteq \n^l(\ker(h)+\n^{N_0-s}F) \\
&\subseteq \im(d)+\n^{N_0-s}F \\
&\subseteq \im(d+\delta)+\n^{N_0}F+\n^{N_0-s}F \\
&= \im(d+\delta)+\n^{N_0-s}F.
\end{align*}
Consequently, if $N_0$ is large enough, namely $N_0>l+r+s$, then 
\begin{align*}
\n^{l}\ker(h+\epsilon) & \subseteq (\im(d+\delta)+\n^{N_0-s}F)\cap \ker (h+\epsilon) \\
&= \im(d+\delta)+\n^{N_0-s}F\cap \ker (h+\epsilon) \\
&\subseteq \im(d+\delta)+\n^{N_0-s-r}\ker (h+\epsilon) \\
&\subseteq \im(d+\delta)+\n^{l+1}\ker (h+\epsilon).
\end{align*}
By Nakayama, this implies that $\n^{l}\ker(h+\epsilon) \subseteq \im(d+\delta)$, and thus $l'\le l$. 

Let now $a\in S$ be any element in $\text{Ann}(H_{\m}^i(R/I))=\text{Ann}(H_{\m}^i(R/I)^{\vee})=\text{Ann}\left(\frac{\ker(h)}{\im(d)}\right)$. 
We obtain that 
\begin{equation}\label{array1}
{\begin{aligned}
a \ker(h+\epsilon) &\subseteq a(\ker(h)+\n^{N_0-s}F) \\
&\subseteq \im(d)+\n^{N_0-s}F \\
&\subseteq \im(d+\delta) +\n^{N_0-s}F.
\end{aligned}}
\end{equation}
Hence 
\begin{equation}\label{array2}
{\begin{aligned}
a \ker(h+\epsilon) &\subseteq (\im(d+\delta) +\n^{N_0-s}F)\cap \ker(h+\epsilon) \\
&=\im(d+\delta)+\ker(h+\epsilon)\cap \n^{N_0-s}F \\
&\subseteq \im(d+\delta)+\n^{N_0-s-r}\ker(h+\epsilon) \\
&\subseteq \im(d+\delta)+\n^{l}\ker(h+\epsilon) \\
&\subseteq \im(d+\delta).
\end{aligned}}
\end{equation}
This shows that $a\in \text{Ann}\left(\frac{\ker(h+\epsilon)}{\im(d+\delta)}\right)=\text{Ann}(H_{\m}^i(R/J)^{\vee})=\text{Ann}(H_{\m}^i(R/J))$. We conclude that $\text{Ann}(H_{\m}^i(R/I)) \subseteq \text{Ann}(H_{\m}^i(R/J))$.
The proof of $\text{Ann}(H_{\m}^i(R/I))\supseteq \text{Ann}(H_{\m}^i(R/J))$ follows in a smiliar way by repeating the same arguments of (\ref{array1}) and (\ref{array2}) for $a\in \text{Ann}(H_{\m}^i(R/J))$ and with the interchanges $h \leftrightarrow h+\epsilon$ and $d \leftrightarrow d+\delta$. 
\end{proof}


If $C^{\sbullet}\colon  \cdots \to C^n \to C^{n+1}\to \cdots$ is a complex of $R$-modules, we will denote by $\tau_r C^{\sbullet}$ the complex obtained by truncating $C^{\sbullet}$ at the $r$-th place, that is, $\tau_r C^{\sbullet}\colon  0\to B^{r+1}( C^{\sbullet})\to C^{r+1}\to C^{r+2}\to\cdots$, where $B^{r+1}( C^{\sbullet})=\im(C^{r}\to C^{r+1})$. Then $H^n(\tau_r C^{\sbullet})=H^n(C^{\sbullet})$ for $n>r$ and $H^n(\tau_r C^{\sbullet})=0$ otherwise.


\begin{theorem}\label{isomorphism}
Let $(R,\m)$ be a Noetherian local ring and $I$ be an ideal of $R$ such that $H_{\m}^i(R/I)$ are finitely generated for every $i=0,1,\ldots,p$. There exists $N>0$ with the following property: for every ideal $J$ such that $J \equiv I \bmod \m^N$ and such that $I$ and $J$ have the same Hilbert function, we have that
$$H_{\m}^i(R/I)\cong H_{\m}^i(R/J)\quad \forall i=0,1,\ldots,p.$$
\end{theorem}
\begin{proof}
It is enough to show $H_{\m}^i(R/I)^{\vee}\cong H_{\m}^i(R/J)^{\vee}$ for all $i=0,1,\ldots,p$, which, by local duality, is equivalent to showing that $H^{d-i}(\text{Hom}_S(F_{\sbullet},S))\cong H^{d-i}(\text{Hom}_S(F^{\epsilon}_{\sbullet},S))$ for all $i=0,1,\ldots,p$. In fact, we will show the existence of an isomorphism $$\tau_{d-p-1}\text{Hom}_S(F_{\sbullet},S)\cong \tau_{d-p-1}\text{Hom}_S(F^{\epsilon}_{\sbullet},S).$$ In order to simplify notation, we write
$$
\begin{tikzcd}
\text{Hom}_S(F_{\sbullet},S)\colon  0 \arrow[r] & P_0 \arrow[r, "g_1"] & \cdots \arrow[r] & P_{d-2} \arrow[r, "g_{d-1}"] & P_{d-1} \arrow[r, "g_{d}"] & P_d \arrow[r] & 0
\end{tikzcd}
$$
 and 
$$
\begin{tikzcd}
\text{Hom}_S(F^{\epsilon}_{\sbullet},S)\colon  0 \arrow[r] & P_0 \arrow[r, "g_1+\delta_1"] & \cdots \arrow[r] & P_{d-2} \arrow[r, "g_{d-1}+\delta_{d-1}"] & P_{d-1} \arrow[r, "g_{d}+\delta_{d}"] & P_d \arrow[r] & 0,
\end{tikzcd}
$$
where $\im(\delta_i)\subseteq \n^{N_0}P_{i}$ for every $i$ and the $P_i$ are free $S$-modules. 

We set some notation for the rest of the proof:
\begin{enumerate}
\item denote $l_i=\ell_{\m}^i(R/I)$. Since, by Theorem \ref{PowerM}, for large $N$ we have $\text{Ann}(H_{\m}^i(R/I)) =\text{Ann}(H_{\m}^i(R/J))$, we also have $\ell_{\m}^i(R/J)=l_i$. This means that $\n^{l_i}$ annihilates both ${\ker(g_{d-i+1})}/{\im(g_{d-i})}$ and ${\ker(g_{d-i+1}+\delta_{d-i+1})}/{\im(g_{d-i}+\delta_{d-i})}$ for all $i=0,\ldots,p$.
\item denote $s_n=\text{AR}(\n,\im(g_n)\subseteq P_n)$ and $t_n=\text{AR}(\n,\ker(g_n)\subseteq P_{n-1})$ for every $n$. Then, by Theorem \ref{T-EisenbudComplexPerturbation} (ii) and \cite[Proposition 2.1 (a)]{herzog2016homology}, $\text{AR}(\n,\im(g_n+\delta_n)\subseteq P_n)$ coincides with $s_n$ for every $n=d,\ldots,d-p$ and $\text{AR}(\n,\ker(g_n+\delta_n)\subseteq P_{n-1})$ coincides with $t_n$ for every $n=d,\ldots,d-p+1$.
\end{enumerate}

\vspace{0.2cm}
\emph{Claim:} We claim that, given any $N_1>0$, for $N\gg 0$ there exist maps $\alpha_n\colon P_n\to P_n$, $n=d-1,d-2,\ldots,d-p-1$, such that $\im(\alpha_n)\subseteq \n^{N_1}P_n$ and such that the all squares in the following diagram commute:
$$
\begin{tikzcd}[row sep=large,column sep = large]
P_{d-p-1} \arrow[r, "g_{d-p}"] \arrow[d, "\text{id}+\alpha_{d-p-1}"] & P_{d-p} \arrow[r, "g_{d-p+1}"] \arrow[d, "\text{id}+\alpha_{d-p}"] & \cdots \arrow[r] & P_{d-2} \arrow[r, "g_{d-1}"] \arrow[d, "\text{id}+\alpha_{d-2}"] & P_{d-1} \arrow[r, "g_{d}"] \arrow[d, "\text{id}+\alpha_{d-1}"] & P_d \arrow[r] \arrow[d, "\text{id}"] & 0 \\
P_{d-p-1} \arrow[r, "g_{d-p}+\delta_{d-p}"]                          & P_{d-p} \arrow[r, "g_{d-p+1}+\delta_{d-p+1}"]                      & \cdots \arrow[r] & P_{d-2} \arrow[r, "g_{d-1}+\delta_{d-1}"]                        & P_{d-1} \arrow[r, "g_d+\delta_d"]                              & P_d \arrow[r]                        & 0
\end{tikzcd}.
$$
\begin{proof}
We start by observing that $\im(g_d)=\im(g_d+\delta_d)$. Indeed, we have $\n^{l_0}P_d\subseteq \im(g_d) $ and also $\n^{l_0}P_d\subseteq \im(g_d+\delta_d)$. Since $\im(g_d)\equiv\im(g_d+\delta_d)\bmod\n^{N_0}P_d$ and $N_0$ can be made large enough by making $N$ large enough, we obtain that $\im(g_d)=\im(g_d+\delta_d)$ for $N\gg 0$ (it is enough $N_0\ge l_0$).

We now proceed to prove the claim by induction on $p$. Suppose first that $p=0$. We need to show that there exists $\alpha_{d-1}\colon P_{d-1}\to \n^{N_1}P_{d-1}$ such that $g_d=(g_d+\delta_d)(\text{id}+\alpha_{d-1})$. Let $N$ be large enough so that $N_0\ge N_1+s_d$. Since $\im(g_d)=\im(g_d+\delta_d)$, for every $x$ in $P_{d-1}$ we have 
\begin{align*}
(g_d+\delta_d)(x)-g_d(x)=\delta_d(x)& \in \n^{N_0}P_d\cap \im(g_d+\delta_d) \\ &\subseteq \n^{N_0-s_d}\im(g_d+\delta_d)\\ &\subseteq \n^{N_1}\im(g_d+\delta_d).
\end{align*}
Hence, there exists $\alpha_{d-1}(x)\in \n^{N_1}P_{d-1}$ such that $g_d(x)=(g_d+\delta_d)(x+\alpha_{d-1}(x))$. As $P_{d-1}$ is a free $S$-module, this shows that the desired map $\alpha_{d-1}$ exists.

Suppose now that the result is valid for $p-1$. Let $q=d-p$. Then, by hypothesis, for $N\gg 0$ there are maps $\alpha_{q+1}\colon P_{q+1}\to \n^{N_1+s_q}P_{q+1}$ and $\alpha_q\colon P_q\to \n^{N_1+s_q}P_q$ making the square on the right hand side bellow commute. 
$$
\begin{tikzcd}[row sep=large,column sep = large]
P_{q-1} \arrow[r, "g_{q}"] \arrow[d, "\text{id}+\alpha_{q-1}", dashed] & P_{q} \arrow[r, "g_{q+1}"] \arrow[d, "\text{id}+\alpha_{q}"] & P_{q+1} \arrow[d, "\text{id}+\alpha_{q+1}"] \\
P_{q-1} \arrow[r, "g_{q}+\delta_q"]                                    & P_{q} \arrow[r, "g_{q+1}+\delta_{q+1}"]                      & P_{q+1}                                    
\end{tikzcd}
$$
We need to show that there exists $\alpha_{q-1}\colon P_{q-1}\to \n^{N_1}P_{q-1}$ such that the square on the left hand side also commutes.
Since $g_{q+1}g_q=0$, the commutativity of the square on the right implies that $\im((\text{id}+\alpha_q)g_q)\subseteq \ker(q_{q+1}+\delta_{q+1})$. For $N\gg 0$ we can assume $N_0 \ge N_1+s_q$. Moreover, it is enough to prove the result assuming $N_1$ is such that $N_1+s_q-t_{q+1}\ge l_{d-q}$. For every $x$ in $P_{q-1}$ we observe that
\begingroup
\allowdisplaybreaks
\begin{align*}
(\text{id}+\alpha_q)(g_q(x)) &= g_q(x)+\alpha_q(g(x)) \\
&\in (\im(g_q)+\n^{N_1+s_q}P_q)\cap \ker(q_{q+1}+\delta_{q+1}) \\
& \subseteq (\im(g_q+\delta_q)+\n^{N_0}P_q+\n^{N_1+s_q}P_q)\cap \ker(q_{q+1}+\delta_{q+1}) \\
&=\im(g_q+\delta_q)+\n^{N_1+s_q}P_q\cap \ker(q_{q+1}+\delta_{q+1}) \\
&\subseteq \im(g_q+\delta_q)+\n^{N_1+s_q-t_{q+1}} \ker(q_{q+1}+\delta_{q+1}) \\
&\subseteq \im(g_q+\delta_q).
\end{align*}
\endgroup
It follows that 
\begingroup
\allowdisplaybreaks
\begin{align*}
(\text{id}+\alpha_q)(g_q(x))-(g_q+\delta_q)(x) &= \alpha_q(g_q(x))-\delta_q(x)\\
&\in (\n^{N_1+s_q}P_q+\n^{N_0}P_q)\cap \im(g_q+\delta_q)\\
&= \n^{N_1+s_q}P_q\cap \im(g_q+\delta_q)\\
&\subseteq \n^{N_1} \im(g_q+\delta_q).
\end{align*}
\endgroup
Hence, there exists $\alpha_{q-1}(x)\in \n^{N_1}P_{q-1}$ such that $(\text{id}+\alpha_q)(g_q(x))=(g_q+\delta_q)(x+\alpha_{q-1}(x))$. As $P_{q-1}$ is a free $S$-module, this shows that the desired map $\alpha_{q-1}$ exists.
\end{proof}

Since $\im(\alpha_n)\subseteq \n P_n$, the maps $\text{id}+\alpha_n\colon P_n\to P_n$ are isomorphisms for all $n=d-1,d-2,\ldots,d-p-1$. Therefore, the maps $\text{id}_{P_d},\text{id}+\alpha_{d-1},\ldots, \text{id}+\alpha_{d-p}$ provide us with an isomorphism of complexes $\tau_{d-p-1}\text{Hom}_S(F_{\sbullet},S)\to \tau_{d-p-1}\text{Hom}_S(F^{\epsilon}_{\sbullet},S)$.
\end{proof}



We present some immediate consequences of Theorems \ref{T-LClength} and \ref{isomorphism}. For a finitely generated $R$-module $M$ we will denote $f_{\m}(M)=\inf\{i\colon H_{\m}^i(M)\text{ is not finitely generated}\}$. 

\begin{corollary}\label{corollaryl}
Let $(R,\m)$ be a Noetherian local ring and let $I$ be an ideal of $R$. There exists $N>0$ with the following property: for every ideal $J$ such that $J \equiv I \bmod \m^N$ and such that $I$ and $J$ have the same Hilbert function, one has
\begin{itemize}\setlength{\itemsep}{1pt}
  \setlength{\parskip}{0pt}
  \setlength{\parsep}{0pt}
\item[(i)] $f_{\m}(R/I)\le f_{\m}(R/J)$.
\item[(ii)]  $\emph{\text{depth}}(R/I)= \emph{\text{depth}}(R/J)$. 
\item[(iii)] If $R/I$ is Cohen-Macaulay, then so is $R/J$. 
\item[(iv)] If $R/I$ is generalized Cohen-Macaulay, then so is $R/J$. Moreover, $H^i_{\m}(R/I)\cong H^i_{\m}(R/J)$ for every $i<\dim(R/I)$.
\end{itemize}
\end{corollary}
\begin{proof}
(i) is a direct corollary of Theorem \ref{T-LClength} and (iv) is a particular case of Theorem \ref{isomorphism}.

We now show (ii). Let $t=\text{depth}(R/I)$. As a consequence of Theorem \ref{T-LClength}, for sufficiently large $N$, $H_{\m}^i(R/I)=0$ implies $H_{\m}^i(R/J)=0$. Therefore, $t\le \text{depth}(R/J)$. 
Let $(-)'$ denote the functor $\text{Hom}_S(-,S). $ For simplicity of notation, rewrite the piece 
$$\begin{tikzcd}
 F_{d-t-1}' \arrow[r, "f_{d-t}'"] & F_{d-t}' \arrow[r, "f_{d-t+1}'"] & F_{d-t+1}' \arrow[r, "f_{d-t+2}'"] & F_{d-t+2}' 
\end{tikzcd} $$
of $\text{Hom}_S(F_{\sbullet},S)$ and the piece
$$ \begin{tikzcd}[row sep=large,column sep = large]
  F_{d-t-1}' \arrow[r, "f_{d-t}'+\epsilon_{d-t}'"] & F_{d-t}' \arrow[r, "f_{d-t+1}'+\epsilon_{d-t+1}'"] & F_{d-t+1}' \arrow[r, "f_{d-t+2}'+\epsilon_{d-t+2}'"] & F_{d-t+2}'
\end{tikzcd} $$
of $\text{Hom}_S(F^{\epsilon}_{\sbullet},S)$ respectively as

\begin{equation}\label{complex1}
 \begin{tikzcd}[row sep=large,column sep = large]
  G \arrow[r, "g"] & F \arrow[r, "f"] & H \arrow[r, "h"] & U 
\end{tikzcd}
\end{equation}
and
\begin{equation}\label{complex2}
 \begin{tikzcd}[row sep=large,column sep = large]
  G \arrow[r, "g+\delta"] & F \arrow[r, "f+\epsilon"] & H \arrow[r, "u+\gamma"] & U,
\end{tikzcd}
\end{equation}
where $\im(\delta)\subseteq \m^{N_0}F$, $\im(\epsilon)\subseteq \m^{N_0}H$ and $\im(\gamma)\subseteq \m^{N_0}U$. $N_0$ can be assumed to be large enough, given that $N$ is large enough.
As $H_{\m}^{t-1}(R/I)$ and $H_{\m}^{t-2}(R/I)$ are zero, this means that the cohomologies of $\text{Hom}_S(F_{\sbullet},S)$ at $H$ and $U$ are zero. Therefore, by Theorem \ref{T-EisenbudComplexPerturbation}, $\im(f)^*=\im(f+\epsilon)^*$. Let $s=\text{AR}(\n,\im(f)\subseteq H)$ and $r=\text{AR}(\n,\ker(f)\subseteq F)$. According to Proposition \ref{L-imkerstar}, if $N_0$ is large enough, we also have $\ker(f)\equiv \ker(f+\epsilon) \bmod \n^{N_0-s}F$.

By means of contradiction, suppose $\text{depth}(R/J)>t$, i.e., suppose $\ker(f+\epsilon)=\im(g+\delta)$. Then we have that 
\begin{align*}
\ker(f) & \subseteq \ker(f+\epsilon)+\n^{N_0-s}F \\
&= \im(g+\delta)+\n^{N_0-s}F \\
&\subseteq \im(g)+\n^{N_0}F+\n^{N_0-s}F \\
&= \im(g)+\n^{N_0-s}F.
\end{align*}
Consequently, if $N_0$ is large enough, namely $N_0>r+s$, then 
\begin{align*}
\ker(f) & \subseteq (\im(g)+\n^{N_0-s}F)\cap \ker(f) \\
&= \im(g)+\n^{N_0-s}F\cap \ker (f) \\
&\subseteq \im(g)+\n^{N_0-s-r}\ker(f) \\
&\subseteq \im(g)+\n\ker(f).
\end{align*}
By Nakayama, this implies that $\ker(f)\subseteq \im(g)$ and therefore $H_{\m}^t(R/I)=0$, which is a contradiction. As such, we must have $\text{depth}(R/J)=t$. This finishes the proof of (ii).

Finally, we show (iii). If $N$ is large enough and $R/I$ is Cohen-Macaulay, then, by (ii), $\dim(R/J)=\dim(R/I)=\text{depth }(R/I)= \text{depth }(R/J)$, which means that $R/J$ is also Cohen-Macaulay. The equality $\dim(R/J)=\dim(R/I)$ holds from the hypothesis that $I$ and $J$ have the same Hilbert function.
\end{proof}

We now apply the previous results to perturbations of Buchsbaum rings. Recall that a local ring $(R,\m)$ is \emph{Buchsbaum} if the difference $\ell(R/({\bf x})R)-e({\bf x},R)$, where ${\bf x}$ is a system of parameters of $R$, is an invariant of $R$, that is, does not depend on ${\bf x}$. Among several equivalent characterizations of Buchsbaum rings (see \cite{schenzel1982applicationsdualizing}, \cite{stuckard1978toward} and \cite{stuckard1986buchsbaum}), we will use a characterization from \cite{schenzel1982applicationsdualizing}, which relies on the dualizing complex of $R$. 

Denote by $\mathcal D(R)$ 
the derived category of the category whose objects are complexes of $R$-modules. 
A bounded bellow complex of $R$-modules $D^{\sbullet}$ is called a \emph{dualizing complex of} $R$ if $D^{\sbullet}$ has finite cohomology and if there exists an integer $h$ such that
$$
H^i(\text{Hom}_R(k,D^{\sbullet}))\cong\left\{\begin{array}{ll}
0 & \text{if } i\neq h \\
k & \text{if } i= h
\end{array}\right..
$$
Furthermore, $D^{\sbullet}$ is said to be \emph{normalized} if $h=0$. A normalized dualizing complex of $R$, if it exists, is unique up to isomorphism in $\mathcal D(R)$. If $R$ has a normalized dualizing complex $D^{\sbullet}$, then according to \cite[Theorem 2.3 (v)]{schenzel1982applicationsdualizing}, $R$ is Buchsbaum if and only if $\tau_{-\dim(R)}D^{\sbullet}$ is isomorphic in $\mathcal D(R)$ to a complex of $k$-vector spaces, where $k$ is the residue field of $R$.

\begin{theorem}\label{Buchsbaum}
Let $(R,\m)$ be a Noetherian local ring and $I$ be an ideal of $R$ such that $R/I$ is generalized Cohen-Macaulay. There exists $N>0$ with the following property: for every ideal $J$ such that $J \equiv I \bmod \m^N$ and such that $I$ and $J$ have the same Hilbert function, we have that $R/I$ is Buchsbaum if and only if $R/J$ is Buchsbaum. 
\end{theorem}
\begin{proof}
Since $R/I$ (resp. $R/J$) is Buchsbaum if and only if $\widehat R/\widehat I$ (resp. $\widehat R/\widehat J$) is Buchsbaum, we can assume $R$ is complete. Hence we can resort to the same setting and notation as in the previous proofs. Let $E^{\sbullet}$ be an injective resolution of the regular ring $S$. We observe that $\text{Hom}_S(R/I,E^{\sbullet})$ is a dualizing complex for $R/I$. Indeed, 
\begin{align*}
H^i(\text{Hom}_{R/I}(k,\text{Hom}_S(R/I,E^{\sbullet})))& \cong
H^i(\text{Hom}_S(k,\text{Hom}_S(R/I,E^{\sbullet}))) \\
&\cong H^i(\text{Hom}_S(k,E^{\sbullet})) \cong \text{Ext}_{S}^i(k,S)
\cong\left\{\begin{array}{ll}
0 & \text{if } i\neq d \\
k & \text{if } i= d
\end{array}\right.,
\end{align*}
where $d=\dim(S)$.
Therefore, the complex $\text{Hom}_S(R/I,E^{\sbullet})[d]$, obtained by a shift of $\text{Hom}_S(R/I,E^{\sbullet})$ $d$ steps to the left, is a normalized dualizing complex of $R/I$. As $F_{\sbullet}$ is a free resolution of $R/I$, the complexes $\text{Hom}_S(F_{\sbullet},S)$ and $\text{Hom}_S(R/I,E^{\sbullet})$ isomorphic in $\mathcal D(R)$. Let $d'=\dim(R/I)=\dim(R/J)$. We then have that, in $\mathcal D(R)$,
\begin{align*}
\tau_{-d'}(\text{Hom}_S(R/I,E^{\sbullet})[d]) =
(\tau_{d-d'}\text{Hom}_S(R/I,E^{\sbullet}))[d] 
\cong (\tau_{d-d'}\text{Hom}_S(F_{\sbullet},S))[d].
\end{align*}
Similarly, $\text{Hom}_S(R/J,E^{\sbullet})[d]$ is a normalized dualizing complex for $R/J$ and $$\tau_{-d'}(\text{Hom}_S(R/J,E^{\sbullet})[d])\cong(\tau_{d-d'}\text{Hom}_S(F^{\epsilon}_{\sbullet},S))[d]\quad\text{in }\mathcal D(R).$$
Since $R/I$ is generalized Cohen-Macaulay, i.e., $H_{\m}^i(R/I)$ is finitely generated for all $i=0,1,\ldots,d'-1$, by the proof of Theorem \ref{isomorphism} if $N$ is large enough then there exists an isomorphism $\tau_{d-d'}\text{Hom}_S(F_{\sbullet},S)\to \tau_{d-d'}\text{Hom}_S(F^{\epsilon}_{\sbullet},S))$. We conclude that 
$$\tau_{-d'}(\text{Hom}_S(R/I,E^{\sbullet})[d]) \cong \tau_{-d'}(\text{Hom}_S(R/J,E^{\sbullet})[d]) $$ 
Hence, $\tau_{-d'}(\text{Hom}_S(R/I,E^{\sbullet})[d])$ is isomorphic to a complex of $k$-vector spaces in $\mathcal D(R)$ if and only if $\tau_{-d'}(\text{Hom}_S(R/J,E^{\sbullet})[d])$ is isomorphic in $\mathcal D(R)$ to that same complex of $k$-vector spaces. By \cite[Theorem 2.3 (v)]{schenzel1982applicationsdualizing}, we conclude that $R/I$ is Buchsbaum if and only if so if $R/J$.
\end{proof}

Finally, we show that, under some aditional assumptions,  Serre's properties $(S_n)$ are also preserved under small perturbations which preserve the Hilbert function. We will need the following lemma.

\begin{lemma}\label{L-dimensionstar}
Let $(R,\m)$ be a Noetherian local ring and let $N\subseteq M$ be finitely generated $R$-modules. Then
$$\emph{\text{Ann}}_R(N)^*\subseteq \emph{\text{Ann}}_{\emph{\text{gr}}_{\m}(R)}(N^*)\subseteq\sqrt{\emph{\text{Ann}}_R(N)^*}.$$
In particular, $\dim_R(N)=\dim_{\emph{\text{gr}}_{\m}(R)}(N^*)$.
\end{lemma}
\begin{proof}
It is clear that $$\text{Ann}_R(N)^*=\bigoplus_{n\ge 0}{\frac{\text{Ann}_R(N)\cap \m^n+\m^{n+1}}{\m^{n+1}}} \quad\text{ annihilates } \quad N^*\cong\bigoplus_{n\ge 0}{\frac{N\cap \m^nM}{N\cap \m^{n+1}M}}.$$

It remains to prove the second inclusion. Take $\bar a\in \m^n/\m^{n+1}$ a homogenous element in $\grm(R)$, represented by $a\in\m^n$, which annihilates $N^*$. This means that 
$$a(N\cap \m^kM)\subseteq N\cap \m^{k+n+1}M\quad\text{for all } k\ge 0.$$
Applying this consecutively for $k=0,n+1,2(n+1),\ldots,(t-1)(n+1)$ we obtain 
$$a^tN\subseteq N\cap \m^{t(n+1)}M\quad\text{for all } t\ge 1.$$
In particular, for every $t>\overline{\text{AR}}(\m,N\subseteq M)=:s$ we have that 
$$a^tN\subseteq N\cap \m^{t(n+1)}M\subseteq \m^{tn+t-s}N\subseteq \m^{tn+1}N.$$
Let $N=Rx_1+\cdots Rx_r$. Then, for each $i=1,\ldots,r$, we have $a^tx_i=\sum_{j=1}^{r}{ b_{ij}x_j}$ for some $b_{ij}\in\m^{tn+1}$. Let $B$ be the matrix $(\delta_{ij}a^t-b_{ij})$, where $\delta_{ij}$ is the Kronecker delta function. Denoting by $x$ the vector $(x_1\cdots x_r)^T$, we have $Bx=0$. Since $\text{det}(B)x=\text{adj}(B)Bx=0$, $\text{det}(B)$ is in $\text{Ann}_R(N)$. But, as $a^t\in\m^{tn}$ and $b_{ij}\in\m^{tn+1}$, $\text{det}(B)$ is of the form $a^{tr}+c$, with $c\in \m^{tn(r-1)+tn+1}=\m^{tnr+1}$. Hence, $a^{tr}\in \text{Ann}_R(N)\cap \m^{tnr}+\m^{tnr+1}$, meaning that $\bar a^{tr}\in \text{Ann}_R(N)^*$. This finishes the proof of the second desired inclusion.

We thus have $\sqrt{\text{Ann}_R(N)^*}=\sqrt{\text{Ann}_{\grm(R)}(N^*)}$, from which it follows that 
\begin{equation*}
\begin{split}
\dim_R(N) & =\dim\left(\frac{R}{\text{Ann}_R(N)}\right)=\dim \left( \grm\left(\frac{R}{\text{Ann}_R(N)}\right)\right) =\dim \left(\frac{\grm(R)}{\text{Ann}_R(N)^*}\right)
\\ & =\dim \left(\frac{\grm(R)}{\sqrt{\text{Ann}_R(N)^*}}\right) =\dim \left(\frac{\grm(R)}{\sqrt{\text{Ann}_{\grm(R)}(N^*)}}\right) 
\\ & = \dim \left(\frac{\grm(R)}{\text{Ann}_{\grm(R)}(N^*)}\right)=\dim_{\grm(R)}(N^*). \qedhere
\end{split} 
\end{equation*}
\end{proof}

\begin{theorem}\label{serreperturbation}
Suppose $(R,\m)$ is excellent. Let $I$ be an ideal of $R$ for which $R/I$ is formally equidimensional. There exists $N>0$ with the following property: if $J$ is an ideal of $R$ for which $I\equiv J\bmod \m^N$, $R/J$ is formally equidimensional and $R/J$ has the same Hilbert function as $R/I$, then for every $n\ge 0$ if $R/I$ satisfies Serre's property $(S_n)$ then so does $R/J$.
\end{theorem}
\begin{proof}
As $R$ is excellent, by \cite[Theorem 23.8]{matsumuraBook} $R/I$ satisfies $(S_n)$ if and only  if $\widehat R /\widehat I$ satisfies $(S_n)$. Since $\widehat R /\widehat I$ is equidimensional, by \cite[Proposition 2.11]{mavarbaro2019regularity}, this is equivalent to $$\dim_S\text{Ext}_S^{d-i}(\widehat R /\widehat I,S)\le i-n$$ for every $i=0,\ldots,d-1$, where $d=\dim(S)$. Therefore, it is enough to show that $\dim_S\text{Ext}_S^{d-i}(\widehat R /\widehat J,S)\le \dim_S\text{Ext}_S^{d-i}(\widehat R /\widehat I,S)$ for all $i$. Indeed, by Theorem \ref{T-EisenbudComplexPerturbation}, for $N$ large enough $H^{d-i}(\text{Hom}_S(F_{\sbullet}^{\epsilon},S))^*$ is isomorphic to a subquotient of $H^{d-i}(\text{Hom}_S(F_{\sbullet},S))^*$ for all $i$, hence we have  
\begin{align*}
\dim_S\text{Ext}_S^{d-i}(\widehat R /\widehat J,S) & = \dim_S H^{d-i}(\text{Hom}_S(F_{\sbullet}^{\epsilon},S)) & {} \\
&= \dim_{\gr_{\n}(S)} H^{d-i}(\text{Hom}_S(F_{\sbullet}^{\epsilon},S)) ^*  & \text{(by Lemma \ref{L-dimensionstar})}\\
& \le \dim_{\gr_{\n}(S)} H^{d-i}(\text{Hom}_S(F_{\sbullet},S)) ^* & {} \\
&= \dim_S H^{d-i}(\text{Hom}_S(F_{\sbullet},S)) & \text{(by Lemma \ref{L-dimensionstar})}\\
&= \dim_S\text{Ext}_S^{d-i}(\widehat R /\widehat I,S). & & \qedhere 
\end{align*} 
\end{proof}

\begin{remark}
Write $\widehat R=S/L$, where $(S,\n)$ is a regular local ring and let 
$$\begin{tikzcd}
 F_{\sbullet}\colon  0 \arrow[r]  & F_d \arrow[r, "f_d"] & F_{d-1} \arrow[r, "f_{d-1}"] & \cdots \arrow[r, "f_2"] & F_1 \arrow[r, "f_1"] & F_0 
\end{tikzcd}$$
be a free resolution of $\widehat R/\widehat I$ as an $S$-module. Let $(-)'$ denote the functor $\text{Hom}_S(-,S)$. 

From \cite[Theorems 2.4 and 3.6]{duarte2021betti} and the proofs presented in this section, we see that there exists a specific value for the number $N$, depending only on the numbers $\text{AR}(\n,\im(f_i)\subseteq F_{i-1})$, $\text{AR}(\n,\im(f_i')\subseteq F_i')$, $\text{AR}(\n,\ker(f_i')\subseteq F_{i-1}')$ and $\ell_{\m}^i(R/I)$, for which Theorems \ref{T-LClength}, \ref{PowerM}, \ref{isomorphism} and \ref{Buchsbaum} hold. While this relation can be made more explicit, we decided not to pursue this direction in this paper, leaving it to the interested reader to analyse the specific details. 
\end{remark}

The following example shows that the hypothesis that the Hilbert function is preserved cannot be ommited in the statements of Theorems \ref{T-LClength}, \ref{PowerM} and \ref{isomorphism}.
\begin{example}
Consider the ring $R=k\llbracket x,y,z \rrbracket$, which has maximal ideal $\m=(x,y,z)$, and the ideal $I=(x^2,y)$. Then $R/I$ is Cohen Macaulay of dimension one. For each $N>2$, we consider the ideal $J_N=(x^2,xy,y-z^N)$, which satisfies $J\equiv I\bmod \m^N$. Then, since $x\not\in J_N$ satisfies $x\m^N\subset J_N$, $R/J_N$ has depth zero. Thus $H^0_{\m}(R/J_N)\neq 0$, while $H^0_{\m}(R/I)=0$.

Notice that $xz^N\in J_N^* \backslash I^*$, so that by \cite[Proposition 3.2]{duarte2021betti} $I$ and $J_N$ do not have the same Hilbert function.
\end{example}



\textbf{Acknowledgements:} I want to thank my advisors Alessandro De Stefani and Maria Evelina Rossi for their guidance and helpful suggestions during this project. I am also grateful to Linquan Ma for several insightful discussions, which led to the development of Lemma \ref{L-NumberGeneratorsInequality} and Theorem \ref{serreperturbation}, and to Mark Walker for pointing out that the map of complexes in the proof of Theorem \ref{isomorphism} is actually an isomorphism. Finally, I thank the department of Mathematics of the University of Genova for supporting my PhD program. 

\bibliographystyle{plain}
\bibliography{citations}
\end{document}